\documentclass[10pt,a4paper,3p]{elsarticle}

\usepackage[cp1250]{inputenc}
\usepackage{fancyhdr}
\usepackage{amssymb,amsmath,amsfonts,theorem,dsfont,a4wide}
\usepackage[T1]{fontenc}
\usepackage{epsfig}

\usepackage{color}
\usepackage[normalem]{ulem}

\usepackage{float}
\usepackage[english]{babel}
\usepackage{graphicx}
\newtheorem{Lemma}{Lemma}[section]
\newtheorem{Theorem}{Theorem}[section]
\newtheorem{Corollary}{Corollary}[section]
\newtheorem{Conjecture}{Conjecture}[section]

\newtheorem{Observation}{Observation}[section]

\newenvironment{proof}
   {\begin{trivlist}\item[]\textbf{\bf{Proof }}\ignorespaces}
   {\qed\end{trivlist}}

\begin{document}

\begin{frontmatter}

\title{On the facial Thue choice number of plane graphs \\ via entropy compression method}

\author[focal]{Jakub Przyby\l o\fnref{PL}}
\ead{przybylo@wms.mat.agh.edu.pl}

\author[rvt]{Jens Schreyer}
\ead{jens.schreyer@tu-ilmenau.de}

\author[els]{Erika \v Skrabu\v l\'akov\'a\fnref{SK}}
\ead{erika.skrabulakova@tuke.sk}

\fntext[PL]{Research partially supported by the Polish Ministry of Science and Higher Education.}
\fntext[SK]{This work was supported by the Slovak Research and Development Agency under the contract No. APVV-0482-11 and by the grant VEGA 1/0130/12.}

\address[focal]{AGH University of Science and Technology, Faculty of Applied Mathematics, al. A. Mickiewicza 30, 30-059 Krakow, Poland}
\address[rvt]{Ilmenau University of Technology, Faculty of Mathematics and Natural Sciences, Institute of Mathematics, Weimarer Str. 25, 98693 Ilmenau, Germany}
\address[els]{Technical University of Ko\v sice, Faculty of Mining, Ecology, Process Control and Geotechnology, Institute of Control and Informatization of Production Processes, B. N\v emcovej 3, 04200 Ko\v sice, Slovakia}

\begin{abstract}
Let $G$ be a plane graph. 
A vertex-colouring $\varphi$ of $G$ is called {\em facial non-repetitive} if for no sequence $r_1 r_2 \dots r_{2n}$, $n\geq 1$, of consecutive vertex colours of any facial path it holds $r_i=r_{n+i}$ for all $i=1,2,\dots,n$. A plane graph $G$ is {\em facial non-repetitively $l$-choosable} if for every list assignment $L:V\rightarrow 2\sp{\mathbb{N}}$ with minimum list size at least $l$ there is a facial non-repetitive vertex-colouring $\varphi$ with colours from the associated lists. \\ 
The {\em facial Thue choice number}, $\pi_{fl}(G)$, of a plane graph $G$ is the minimum number $l$ such that $G$ is
facial non-repetitively $l$-choosable. \\
%In this article we
We use the so-called entropy compression method to show that $\pi_{fl} (G)\le c \Delta$ for some absolute constant $c$ and $G$ a plane graph with maximum degree $\Delta$. Moreover, we give some better (constant) upper bounds on $\pi_{fl} (G)$ for special classes of plane graphs.
\end{abstract}


\end{frontmatter}

\section{Introduction}
%%%%%%%%%%%%%%%%%%%%%%

\hspace{0,4cm} In this paper we consider simple, finite, plane and undirected graphs. A finite sequence $R=r_1 r_2 \dots r_{2n}$ of symbols is called a {\em repetition} if $r_i=r_{n+i}$ for all $i=1,2,\dots,n$. A sequence $S$ is called {\em repetitive} if it contains a subsequence of consecutive terms that is a repetition. Otherwise $S$ is called  {\em non-repetitive}. Let $\varphi$ be a colouring of the vertices of a graph $G$.  We say that $\varphi$ is {\em non-repetitive} if for any simple path on vertices $v_1$, $v_2, \dots, v_{2n}$ in $G$ the associated sequence of colours $\varphi(v_1)$ $\varphi(v_2) \dots \varphi(v_{2n})$ is not a repetition. The minimum number of colours in a non-repetitive vertex-colouring of a graph G is the {\em Thue chromatic number}, $\pi(G)$.

Non-repetitive sequences have been studied since the beginning of the 20$\sp{th}$ century. In 1906 the Norwegian mathematician and number theoretician Axel Thue gave rise to a systematic study of \emph{Combinatorics on Words}, or \emph{Symbolic Dynamics}. %combinatorial structure in words.
Thue \cite{Th06} showed then that there are arbitrarily long non-repetitive sequences over three symbols,
a result that was rediscovered several times in different branches of mathematics, see e.g.~\cite{Allouche,Bean,Lothaire}.
Thue's result in the language of graph theory states:

\begin{Theorem}[\cite{Th06}]\label{Thue}
Let $P_n$ be a path on $n$ vertices. Then $\pi(P_1)=1$, $\pi(P_2)=\pi(P_3)=2$ and for $n>3$, $\pi(P_n)=3$.
\end{Theorem}

Non-repetitive sequences regarding graphs occurred for the first time in 1987 (see \cite{Cur87}).
In 1993 Currie \cite{Cur02} showed that for every cycle of length $n\in\{5$, $7$, $9$, $10$, $14$, $17\}$, $\pi(C_n)=4$ and  for other lengths of cycles on at least $3$ vertices, $\pi(C_n)=3$. However, the motivation to study non-repetitive colourings of \emph{graphs} themselves comes from the paper of Alon et al. \cite{AGHR02} from 2002. The upper bound of  the Thue chromatic number  of the form $\pi(G)\leq C\Delta^2$ (where $\Delta$ is the maximum degree of $G$ and $C$ is a constant) was  for all graphs $G$ settled in \cite{Gr_new_1}. The proof was based on Lov\'asz Local Lemma, what  immediately yields the same result for the naturally defined list version of the problem, where the corresponding parameter is called the \emph{nonrepetitive choice number} and denoted by $\pi_{\rm ch}(G)$, see e.g.~\cite{Dujmovic} for details.
The constant $C$ was then pushed in several consecutive papers \cite{Gr_new_2,HaJe12} roughly down to 12.92 by means of the same tool.
Quite recently however, Dujmovi\'c et al. %\cite{Dujmovic}
used the so-called entropy compression method to prove the following asymptotical improvement: 
\begin{Theorem}[\cite{Dujmovic}]\label{DujmobicetalTh}
For every graph $G$ with maximum degree $\Delta$,
$\pi_{\rm ch}(G)\leq (1+o(1))\Delta^2$.
\end{Theorem}
By \cite{AGHR02} this result is known to be tight up to the logarithmic factor in $\Delta$.
This bound can however be improved (in the non-list version) if we limit ourselves to special families of graphs.
It is for instance known that $\pi(T)\leq 4$ for every tree $T$, see e.g. \cite{Gr_new_2}.
Perhaps one of the most interesting open questions in this field comes from~\cite{AGHR02} and concerns determining whether $\pi(G)$ is bounded (from above) in the class of planar graphs.
Surprisingly, the answer to the same question but in the list setting is negative, even for the family of trees, see~\cite{MendezZhu}.
%On the other hand, $\pi_{\rm ch}(P_n)$ the answer to this question
%Various other results and open problems might be found e.g. in ~\cite{Gr_new_2,Gr07}.
%In \cite{Gr07}
Various questions concerning non-repetitive colourings of graphs were formulated e.g. in ~\cite{CzGr07,Gr_new_2,Gr07}, and consequently numerous their variations have appeared in the 
literature.

We further develop the study of planar (plane in fact) graphs
and analogously as in \cite{BaCz11}, \cite{HaJe12}, \cite{HJSS09}, \cite{JeSk09}, \cite{JP12} and  \cite{JSES11}
we introduce a facial non-repetitive colouring as follows: Let $G$ be a 
plane graph. A {\em facial path} in $G$ is a path of $G$ on consecutive vertices (and edges) on the boundary walk of some face (of any component) of $G$. Let $\varphi$ be a colouring of the vertices of 
$G$.
We say that $\varphi$ is {\em facial non-repetitive} if for every 
facial path on vertices $v_1$, $v_2, \dots, v_{2n}$ in $G$ the associated sequence of colours $\varphi(v_1)$ $\varphi(v_2) \dots \varphi(v_{2n})$ is not a repetition. The minimum number of colours in a facial non-repetitive vertex-colouring of a graph $G$ is the {\em facial Thue chromatic number}, $\pi_f(G)$.
If additionally the colour of every vertex $v$ is chosen from a list $L(v)$ of admissible possibilities assigned to $v$, we speak about a {\em facial non-repetitive list colouring} $\varphi_L$ of the graph $G$ with list assignment $L$.
The graph $G$ is {\em facial non-repetitively $l$-choosable} if for every list assignment
$L:V\rightarrow 2\sp{\mathbb{Z}_+}$ (where $\mathbb{Z}_+$ - the set of positive integers, is chosen for further convenience) with minimum list size at least $l$ there is a facial non-repetitive vertex-colouring $\varphi_L$ with colours from the associated lists. The {\em facial Thue choice number}, $\pi_{fl}(G)$, of a plane graph $G$ is the minimum number $l$, such that $G$ is facial non-repetitively $l$-choosable.
Note that both these parameters, 
$\pi_{f}(G)$ and $\pi_{fl}(G)$,
are defined only for planar graphs and depend on embedding of the graph (thus for plane graphs in fact).

Harant and Jendro\v l \cite{HaJe12} asked whether there is a constant bound for the facial Thue chromatic number of all plane graphs. Recently Bar\'at and Czap %\cite{BaCz11}
proved the following.
\begin{Theorem}[\cite{BaCz11}]\label{BaratCzapTh}
For any plane graph $G$, $\pi_f(G)\leq 24$.
\end{Theorem}
Whether the facial Thue choice number of plane graphs could be bounded from above by a constant is still an open question.
Note however that obviously, for every plane graph $G$,
\begin{equation} \label{3piEq}
\pi_{f}(G)\leq \pi_{fl}(G)\leq \pi_{\rm ch}(G),
\end{equation}
cf. Theorems~\ref{DujmobicetalTh} and~\ref{BaratCzapTh}.

In this paper we use the entropy compression method to show an upper bound for the facial Thue choice number of an arbitrary plane graph $G$ depending linearly on its maximum degree $\Delta$. In particular we prove 
the bound of the form $\pi_{fl}(G)\leq 5 \Delta$ and its asymptotical improvement,
$\pi_{fl}(G)\leq (2+o(1)) \Delta$, see Theorem~\ref{mainresult} in Section~\ref{main_result_section}.
We also supplement these results by
showing constant upper bounds on $\pi_{fl}(G)$ for some classes of plane graphs in Section~\ref{families_section}.

\section{A general bound for plane graphs\label{main_result_section}}
%%%%%%%%%%%%%%%%%%%%%%%%%%%%%%%%%%%%%%%%%

\begin{Theorem}\label{mainresult}
If $G$ is a plane graph of maximum degree $\Delta$, then: \\
%\begin{eqnarray}
$\pi_{fl}(G) \leq 5\Delta$, \\
$\pi_{fl}(G) \leq (2+o(1))\Delta$.
%\end{eqnarray}
\end{Theorem}

\begin{proof}
Let $G=(V,E,F)$ be an arbitrarily fixed plane graph with vertex set $V$, edge set $E$ and face set $F$. Furthermore, let $L:V\rightarrow 2^{\mathbb{Z}_+}$ be its fixed list assignment, where $|L(v)|=l$ for $v\in V$.
Order (label) the vertices arbitrarily into a sequence $v_1,v_2,\ldots,v_n$,
and analogously fix any linear ordering in every list $L(v)$, $v\in V$.
By a \emph{partial colouring} of $G$ from $L$ we shall mean an assignment $c:V\to\mathbb{Z}_+\cup\{0\}$
such that $c(v)\in L(v)\cup\{0\}$ for $v\in V$, where the vertices with `$0$' assigned will be identified with non-coloured ones
(thus, e.g., a graph whose none of the vertices has a colour chosen from its list has a partial colouring with all vertices assigned with `$0$').
Thus we will only be further concerned with repetitions generated by positive integers.\\

\textbf{The algorithm} \\

Suppose we use the following randomized algorithm in order to find a desired vertex colouring.
In the first step, we choose randomly an element from $L(v_1)$, each with equal probability of $1/l$, and set it as $v_1$'s colour.
Then, in the second step, we first do the same for $L(v_2)$ and $v_2$.
If $v_1$ and $v_2$ are the consecutive vertices of a boundary walk of some face and their colours are the same, then we uncolour $v_2$ and move to the next step of the construction.
Suppose now that after $t$ steps of the algorithm no facial path of $G$ is repetitively coloured, and there are still some vertices which are not coloured. Then in the consecutive step we choose $v_i$ with the smallest $i$ among these unpainted vertices, and choose (independently and equiprobably) a random colour from $L(v_i)$ for it. If as a result a facial path $P$ which makes up a repetition is created, then we uncolour the vertices of this half of $P$ which contains $v_i$ (if more than one such $P$ appears we choose any of this according to some previously specified rule), thus `destroying' all possible repetitions (all including the colour of $v_i$) on the faces of $G$, and proceed with the next step. We continue until all vertices are painted after some step, in which case we obtain a desired colouring. \\
So that we are certain that this algorithm eventually stops, we also bound the number of admitted steps by $T$, where $T$ is some large integer (much larger than $n$). In fact it shall be clear from our further reasoning that the probability that such algorithm stops after less than $T$ steps converges to $1$ as $T$ tends to infinity.  \\

\textbf{The analysis} \\

For a fixed (large) $T$, suppose that the probability that our algorithm performs all $T$ admitted steps while trying to find an appropriate
colouring of our (fixed) graph $G$ on $n$ vertices is equal to $1$.
Since the lists $L(v)$ are linearly ordered, then picking a colour is in each of its steps equivalent to an independent random choice of an integer in $[1,l]:=\{1,2,\ldots,l\}$.
Thus there are exactly $l^T$ distinct executions of the algorithm (each equally probable).
We will show that counting these otherwise we might compress this range of varieties into a smaller extent (this explains the term ``entropy compression'' used for the name of the method applied).
In other words, we will obtain a contradiction by a double counting argument, and consequently we will prove that at least one possible execution of the algorithm must stop after less than $T$ steps, and thus yield a desired colouring of $G$.

To perform the alternative counting, suppose that at the end our algorithm returns a \emph{record} containing some information on its execution.
This will be a $T$-element vector denoted by $R$ whose every position corresponds to a single consecutive step of the algorithm.
Suppose we coloured a vertex $v$ in a given step $t$.
Then we set $R(t)=\emptyset$ if no repetition was created in this step. Otherwise, if a path $P$ was a repetition whose appropriate half of colours had to be erased, we set $R(t)=r(v,P)$, where $r(v,P)=(h,q,o)$ is defined as follows.
We set $h$ to be the size of the repetition (hence $P$ has order $2h$) and $q$ to be the distance of $v$ along $P$ to the closer end of $P$
%say $u$,
(hence $q\in [0,h-1]$). Finally, we define $o\in[1,2\Delta]$, which corresponds to an ``orientation'' of $P$ with respect to $v$, in the following manner:
First note that for fixed $v$, $h$ and $q$ as above, there are at most $2\Delta$ \textbf{facial} paths of length $2h$ which contain $v$ and whose one end is at distance exactly $q$ from $v$ along the path. For every fixed $v$, $h$, $q$ we may thus assume that before having launched the algorithm, we had chosen an ordering of such facial paths, and thus (during an execution of the algorithm) we may set $o$ to be the index of $P$ in this ordering. Note that then, given $v,h,q$ and $o$, we may unambiguously reconstruct the unique path $P$ for which $r(v,P)=(h,q,o)$. The following is the key observation responsible for a one-to-one correspondence between two ways of counting. \\

\begin{Observation}\label{key-one-to-one-obs}
For every record $R$ and a (partial) vertex colouring $c$ of $G$ from $L$ there is at most one possible execution of the algorithm which returns $R$ and $c$.
\end{Observation}

\begin{proof}
For any execution of the algorithm resulting in creating our fixed record $R$, let $Y_1,Y_2,\ldots,Y_T$ be a sequence where $Y_t$ is the set of indices of yet not coloured vertices at the beginning of step $t$, $t\in[1,T]$, hence $Y_1=[1,n]$. By analyzing the vector $R$ from left to right, one may easily see that the sequence $Y_1,Y_2,\ldots,Y_T$ is unique (if it exists) for the fixed $R$. Indeed, for $t=1,2,\ldots,n-1$, if $R(t)=\emptyset$, then $Y_{t+1}=Y_t\smallsetminus\{j\}$ with $j=\min Y_t$. If on the other hand, $R(t)=(h,q,o)$, then we know that in step $t$ the vertex $v_j$ with $j=\min Y_t$ was assigned a colour, and consequently a repetitively coloured path containing $v_j$ was created. As mentioned above, there is a unique such encoded path, say $P$ (with $r(v_j,P)=(h,q,o)$).
Suppose that $v_{i_1},v_{i_2},\ldots,v_{i_{2h}}$ are the consecutive vertices of $P$ with $v_{i_{q+1}}=v_j$ ($q+1\leq h$), then $Y_{t+1}=Y_t\cup\{i_1,i_2,\ldots,i_h\}$.

We will use the fact above to prove that given $R$ and a partial vertex colouring $c$ we may reconstruct all consecutive choices (if these exist) that led to obtaining $R$ and $c$, and hence prove a uniqueness of the corresponding algorithm execution. It is sufficient to analyze the vector $R$ from right to left now, modifying at the same time the partial vertex colouring $c$ (so that it always corresponds to a colouring that was obtained after the step that we are about to analyze at the moment) until none of the vertices is coloured under $c$.
Suppose we were able to reconstruct all choices that had to be undertaken by the algorithm in steps $T,T-1,\ldots,t$ for some $t\in[2,T+1]$
(where we set $t=T+1$ for convenience before analyzing the first step $T$).
Let $j=\min Y_{t-1}$. Then we know that the vertex $v_j$ had its colour randomly chosen from its list in step $t-1$.

Assume first that $R(t-1)=\emptyset$. Then we know from the contemporary (partial) colouring $c$, which colour had to be chosen for $v_j$ (i.e., the colour $c(v_j)$). We thus modify our colouring by erasing the colour of $v_j$ (i.e., setting $c(v_j)=0$) to obtain the colouring produced by the algorithm after step $t-2$, and continue.

Assume then that $R(t-1)=(h,q,o)$. Let $v_{i_1},v_{i_2},\ldots,v_{i_{2h}}$ with $v_{i_{q+1}}=v_j$ ($q+1\leq h$) be the consecutive vertices of the unique path $P$
such that $r(v_j,P)=(h,q,o)$. Then, since our contemporary colouring $c$ corresponds to the one after the end of step $t-1$ (hence the vertices $v_{i_{h+1}},\ldots,v_{i_{2h}}$ are coloured under it, while $v_{i_1},\ldots,v_{i_{h}}$ are not) and the vertices of $P$ made up a repetition in step $t-1$, we know that the colour equal to $c(v_{i_{q+1+h}})$ had to be chosen for $v_j=v_{i_{q+1}}$ in step $t-1$. Additionally, to obtain the colouring produced by the algorithm after step $t-2$, we set $c(v_{i_{p}}):=c(v_{i_{p+h}})$ for $p\in[1,h]$, $p\neq q+1$ (not changing the colours of the remaining vertices). \\
\end{proof}

\textbf{The counting} \\

Let us denote by $\mathcal{R}_T$ the set of records produced by the algorithm in all possible executions with exactly $T$ steps admitted.
We first bound $|\mathcal{R}_T|$ from above.
For any numbers $n_1,n_2,\ldots,n_k$ and nonnegative integers $p_1,p_2,\ldots,p_k$,
we will denote by $n_1^{p_1}n_2^{p_2}\ldots n_k^{p_k}$ the sequence
$$(\underbrace{n_1\ldots,n_1}_{p_1},\underbrace{n_2\ldots,n_2}_{p_2},\ldots,\underbrace{n_k,\ldots,n_k}_{p_k}).$$
Moreover, we will write $n_i$ instead of $n_i^1$ and assume that $n_i^0$ contributes no symbol to the sequence.
For any two sequences $X=(x_1,\ldots,x_b)$ and $Z=(z_1,\ldots,z_d)$, we shall also write $X\oplus Z$
to denote their \emph{concatenation}, i.e.,
$X\oplus Z = (x_1,\ldots,x_b,z_1,\ldots,z_d)$.
For any given record $R$, and $t\in[1,T]$, let $S(R(t))=1$ if $R(t)=\emptyset$, or let $S(R(t))=1(-1)^h$ if $R(t)=(h,q,o)$.
Note that multiplicities of $1$'s and $-1$'s correspond to the number of coloured and uncoloured, resp., vertices in the given steps of the algorithm.
Then let $$S'(R)=S(R(1))\oplus S(R(2))\oplus\ldots\oplus S(R(T)).$$
Since the number of coloured vertices cannot be negative, and each step contributes exactly one `$1$' to the sequence $S'(R)$, its length does not exceed $2T$ (and in fact is always smaller). It shall be more convenient for us to consider sequences of length exactly $2T$, hence we set $S(R):=S'(R)\oplus 1^w$, where $w$ is the difference between $2T$ and the length of $S'(R)$ (hence the sum of the elements of $S(R)$ is now equal to $2w$).\\

Denote by $\mathcal{S}_{2T}$ the set of all sequences $S(R)$ which might be returned by our algorithm with $T$ steps admitted,
and let $\mathcal{A}_m$ stand for the set of all $m$-element sequences consisting of $1$'s and $-1$'s (hence $|\mathcal{A}_m|=2^m$).
Now set $\mathcal{A}=\bigcup_{m\in\mathbb{N}}\mathcal{A}_m$ and let $f:\mathcal{A}\to \mathbb{N}$ be a function defined as follows: \\
For any sequence $S$ which contains at least one `$-1$', i.e., a sequence of the form:
$$S=1^{k_1}(-1)^{h_1}1^{k_2}(-1)^{h_2}\ldots(-1)^{h_p}1^{k_{p+1}}\in \mathcal{A}_m$$
with $m,p\geq 1$, $h_i>0$ for $i=1,\ldots,p$, $k_i>0$ for $i=2,\ldots,p$ and $k_1,k_{p+1}\geq 0$,
$$f(S):=2\Delta h_1\cdot 2\Delta h_2\cdot\ldots\cdot 2\Delta h_p,$$
and $f(1^m):=1$. \\
Note that since for every element of a record of the form $(h,q,o)$, there are exactly $h$ possible values of $q$ and $2\Delta$ possible ones of $o$ for a fixed $h$, then $f(S)$ bounds from above the number of possible records $R$ with $S(R)=S$. (It follows also by our hypothesis assumption that all executions of the algorithm consist of exactly $T$ steps.)
Therefore,
\begin{equation}\label{suminequality1}
|\mathcal{R}_T|\leq \sum_{S\in \mathcal{S}_{2T}} f(S) \leq \sum_{S\in \mathcal{A}_{2T}} f(S).
\end{equation}
We shall thus investigate the following sequence:
\begin{equation}\label{suminequality2}
a_m: = \sum_{S\in \mathcal{A}_{m}} f(S).
\end{equation}

To calculate the value of $a_m$, we use the fact that the following recurrence relation holds for $m\geq 2$ by the definition of this sequence:
\begin{equation}\label{sequence_inequality_1}
a_m=a_{m-1}+2\Delta a_{m-2}+2\Delta\cdot 2a_{m-3}+\ldots + 2\Delta(m-2)a_1+2\Delta(m-1)+2\Delta m,
\end{equation}
where the factor $a_{m-1}$ corresponds to the sequences $S\in \mathcal{A}_{m}$ with `$1$' at the end,
the factor $2\Delta a_{m-2}$ - to the ones with the end of the form `$\ldots 1,-1$',
the factor $2\Delta \cdot 2a_{m-3}$ - to the ones with the end of the form `$\ldots 1(-1)^2$',$\ldots$,
the factor $2\Delta(m-2)a_1$ - to the ones with the end of the form `$\ldots 1(-1)^{m-2}$',
the factor $2\Delta(m-1)$ - to the sequence $1(-1)^{m-1}$, and
the factor $2\Delta m$ - to the sequence $(-1)^{m}$.

Observe now that by (\ref{sequence_inequality_1}), for $m\geq 3$:
$$a_m-a_{m-1}=a_{m-1}+(2\Delta-1)a_{m-2}+2\Delta a_{m-3}+2\Delta a_{m-4}+\ldots+2\Delta a_1 +4\Delta,$$
hence
\begin{equation}\label{sequence_inequality_2}
a_m=2a_{m-1}+(2\Delta-1)a_{m-2}+2\Delta a_{m-3}+2\Delta a_{m-4}+\ldots+2\Delta a_1 +4\Delta.
\end{equation}
Consequently, by (\ref{sequence_inequality_2}), for $m\geq 4$:
$$a_m-a_{m-1}=2a_{m-1}+(2\Delta-3)a_{m-2}+a_{m-3},$$
and thus:
\begin{equation}\label{recurrenceequation}
a_m=3a_{m-1}+(2\Delta-3)a_{m-2}+a_{m-3},
\end{equation}
while $a_1=2\Delta+1$, $a_2=8\Delta+1$ and $a_3=4\Delta^2+20\Delta+1$. \\

\textbf{%Proof of Theorem~\ref{mainresult}:} \\
The proof} \\

For $\Delta=1$ the theorem is trivial, while for $\Delta=2$ it follows from the fact that $\pi_{fl}(P_n)\leq 4$ and $\pi_{fl}(C_n)\leq 5$ for every $n$, see Corollary~\ref{GPZ} and Theorem~\ref{cycle} in Section~\ref{families_section}.

For any fixed $\Delta \geq 3$ we will prove that at least one possible execution of the algorithm must stop after less than $T$ steps.
As mentioned, we proceed by the way of contradiction and thus suppose that all executions of the algorithm perform all $T$ admitted steps
(hence there are $l^T$ possible executions of the algorithm),
where $T$ is some `large' constant (dependent on $n$ and $\Delta$).

Assume first that $\Delta=3$ and the lists length $l$ equals to $15=5\Delta$. Then by~(\ref{recurrenceequation}),
$$a_m=c_0\lambda_0^m+c_1\lambda_1^m+c_2\lambda_2^m,$$
where $c_0,c_1,c_2\in\mathbb{C}$ are some (fixed) constants and $\lambda_0,\lambda_1,\lambda_2$
are the roots of
the characteristic equation\footnote{A polynomial used to solve recurrence relation, see \cite{Ch08}.}
$\lambda^3-3\lambda^2-3\lambda-1=0$
in the complex domain $\mathbb{C}$, i.e.,
$\lambda_0\approx 3.847$ ($\lambda_0<3.85$), $\lambda_1\approx -0.424+0.284i$, $\lambda_2\approx -0.424-0.284i$.
These might be precisely calculated using e.g.
Cardano's formula\footnote{By Cardano's formula, $\lambda_0=\sqrt[3]{2}+\sqrt[3]{4}+1$, $\lambda_1=1-\frac{1}{2}(\sqrt[3]{2}+\sqrt[3]{4})+\frac{\sqrt{3}}{2}(\sqrt[3]{4}-\sqrt[3]{2})i$ and $\lambda_2=\overline{\lambda_1}$.} (see e.g. \cite{Co08})
or a computer program. Since $|\lambda_1|<\lambda_0$ and $|\lambda_2|<\lambda_0$, then there exists a positive constant $C$ so that
$$a_m\leq C\lambda_0^m<C\cdot 3.85^m.$$
Therefore, by~(\ref{suminequality1}) and~(\ref{suminequality2}),
$$
|\mathcal{R}_T| < C\cdot (3.85)^{2T} = C\cdot 14.8225^T.
$$
Since there are at most $16^n$ possible partial vertex colourings for a graph of order $n$ from the given $15$-element lists (counting in the $16$-th possibility that a vertex is not coloured) and since by Observation~\ref{key-one-to-one-obs} every pair consisting of a possible record and a partial colouring of $G$ corresponds to at most one algorithm execution, then there should be at most
$$16^n\cdot C\cdot 14.8225^T$$
distinct executions of the algorithm (where $n$ is some fixed integer: we analyze a fixed $G$). Therefore, for sufficiently large $T$, we obtain a contradiction with the assumption that there are exactly
$15^T$ distinct executions of the algorithm.

Assume now that $\Delta\geq 4$ is fixed and
$$l=\left\lceil \left[1+2\sqrt{\frac{2\Delta}{3}}\cos \left(\frac{\arccos\left(\sqrt{\frac{27}{8\Delta}}\right)}{3}\right)\right]^2+0.5\right\rceil.$$
Denote
$$\varphi=\frac{\arccos\left(\sqrt{\frac{27}{8\Delta}}\right)}{3}$$
and first note that
$$l=\left\lceil \left(1+2\sqrt{\frac{2\Delta}{3}}\cos \varphi\right)^2+0.5\right\rceil \leq  \left\lceil  \left(1+2\sqrt{\frac{2\Delta}{3}}\right)^2+0.5\right\rceil <\left(1+2\sqrt{\frac{2\Delta}{3}}\right)^2+1.5<5\Delta,$$
where the last inequality might be verified by elementary calculations (for $\Delta\geq 4$).
On the other hand,
\begin{eqnarray*}
l &\leq& \left(1+2\sqrt{\frac{2\Delta}{3}}\cos \varphi\right)^2+1.5\\
  &=& \frac{8\Delta}{3}\cos^2 \varphi +4\sqrt{\frac{2\Delta}{3}}\cos \varphi+2.5\\
  &=& \left[2+\left(\frac{8}{3}\cos^2\varphi-2+4\sqrt{\frac{2}{3\Delta}}\cos \varphi+\frac{2.5}{\Delta} \right)\right]\Delta,
\end{eqnarray*}
where
\begin{eqnarray*}
 & & \lim_{\Delta\to\infty} \left(\frac{8}{3}\cos^2\varphi-2+4\sqrt{\frac{2}{3\Delta}}\cos \varphi+\frac{2.5}{\Delta}\right)\\
 &=& \lim_{\Delta\to\infty} \frac{8}{3}\cos^2\left(\frac{\arccos\left(\sqrt{\frac{27}{8\Delta}}\right)}{3}\right)-2 \\
 &=& \frac{8}{3}\cos^2\left(\frac{\frac{\pi}{2}}{3}\right)-2 = \frac{8}{3}\left(\frac{\sqrt{3}}{2}\right)\sp{2}-2= 2-2 = 0.
\end{eqnarray*}
We thus have that $l=(2+o(1))\Delta$.\\ %, hence $l$ is in fact essentially of order $2\Delta$. \\
Now let us get back to estimating $|\mathcal{R}_T|$ in this case, hence we shall analyze the recurrence relation $a_m=3a_{m-1}+(2\Delta-3)a_{m-2}+a_{m-3}$ by solving its
characteristic equation
$\lambda^3-3\lambda^2-(2\Delta-3)\lambda-1=0$.
The substitution $\lambda = y-\frac{-3}{3\cdot 1}=y+1$ will reduce this equation to the form $y\sp{3}+P\cdot y + Q=0$ where $P=Q=-2\Delta$. Hence the discriminant
$$\frac{P\sp{3}}{27}+\frac{Q\sp{2}}{4}=\frac{-8\Delta\sp{3}}{27}+\frac{4\Delta\sp{2}}{4}=\Delta\sp{2}\left(1-\frac{8\Delta}{27}\right)$$ is for $\Delta\geq 4$ negative and thus by \cite{Co08}
%For this we have:
%$$P=\frac{-(2\Delta-3)}{1}-\frac{(-3)^2}{3\cdot 1^2}=-2\Delta,$$
%$$Q=\frac{2(-3)^3}{27\cdot 1^3}+\frac{-1}{1}-\frac{(-3)[-(2\Delta-3)]}{3\cdot 1^3}=-2\Delta,$$
%$$\Delta'=\left(\frac{P}{3}\right)^3+\left(\frac{Q}{2}\right)^2 = \Delta^2\left(1-\frac{8}{27}\Delta\right)>0,$$
our equation has three real roots of the form
\begin{eqnarray*}
 \lambda'_0 & = & 1+ 2\sqrt{\frac{2\Delta}{3}}\cdot\cos\varphi,\\
 \lambda'_1 & = & 1+ 2\sqrt{\frac{2\Delta}{3}}\cdot\cos\left(\varphi +\frac{2\pi}{3}\right),\\
\lambda'_2 & = & 1+ 2\sqrt{\frac{2\Delta}{3}}\cdot\cos\left(\varphi +\frac{4\pi}{3}\right),
\end{eqnarray*}
where $\varphi \in (0;\frac{\pi}{6})$ was defined above. Hence $|\lambda'_1|<\lambda'_0$ and $|\lambda'_2|<\lambda'_0$.
Consequently,
$$a_m=c'_0(\lambda'_0)^m+c'_1(\lambda'_1)^m+c'_2(\lambda'_2)^m\leq C'(\lambda'_0)^m$$
for some constants (dependent on $\Delta$) $c'_0,c'_1,c'_2$ and $C'>0$.
Therefore, by~(\ref{suminequality1}) and~(\ref{suminequality2}),
$$
|\mathcal{R}_T| \leq C'\cdot (\lambda'_0)^{2T} < C'\cdot ((\lambda'_0)^{2} + 1)\sp{T} < C'\cdot (l-0.5)^T.
$$
By Observation~\ref{key-one-to-one-obs}, we analogously as above obtain that there should be at most
$$(l+1)^n\cdot C'\cdot (l-0.5)^T$$
distinct executions of the algorithm. For sufficiently large $T$, we thus obtain a contradiction with the assumption that there are exactly $l^T$ executions of the algorithm.

In both cases we reach a contradiction with the hypothesis that all executions of the algorithm perform all $T$ admitted steps.
Consequently, at least one possible execution of the algorithm must stop after less than $T$ steps, and thus yield a desired colouring of $G$. \\
\end{proof}

\section{Families of graphs with a constant upper bound for $\pi_{fl}$\label{families_section}}
%%%%%%%%%%%%%%%%%%%%%%%%%%%%%%%%%%%%%%%%%

In this section we provide a number of classes of plane graphs for which constant upper bounds on the investigated parameter hold.
We will start by presenting auxiliary lemmas applicable in the later arguments.\\

\textbf{Preliminary lemmas}\\

The following lemma was proved in \cite{HJSS09}.
For a sequence of symbols $S=a_1$ $a_2\dots a_n$ with $a_i\in\mathbb{A}$, for all $1\leq k\leq l\leq n$, the block $a_k$ $a_{k+1}\dots a_{l}$ is denoted by $S_{k,l}$.

\begin{Lemma}[\cite{HJSS09}]\label{1}
Let $A=a_1 a_2 \dots a_m$ be a non-repetitive sequence with $a_i\in\mathbb{A}$ for all \    $i=1,2,\dots,m$. Let $B^i=b^i_1 b^i_2 \dots b^i_{m_i}$; $0\leq i\leq r+1$, be non-repetitive sequences with $b\sp{i}_j\in\mathbb{B}$ for all \  $i=0,1,\dots,r+1$ and $j=1,2,\dots,m_i$. If $\mathbb{A}\cap\mathbb{B}=\emptyset$, then $S=B^0$ $A_{1,n_1}$ $B^1$ $A_{n_1+1,n_2}\dots B^r$ $A_{n_r+1,m}$ $B^{r+1}$ is a non-repetitive sequence.
\end{Lemma}

Note that a rainbow sequence (a sequence of length $k$ consisting of $k$ different symbols) is trivially non-repetitive. Moreover, if each sequence $B\sp{i}$, $0\leq i\leq r+1$, from Lemma \ref{1} consists of only one element, then it is also trivially non-repetitive and we immediately obtain the following corollary:

\begin{Corollary} \label{2}
Let $A=a_1 a_2 \dots a_m$ be a rainbow sequence with $a_i\in\mathbb{A}$ for all \  $i=1,2,\dots,m$. For $i=0,1,\dots,r+1$ let $B^i\notin\mathbb{A}$. Then $S=B^0$ $A_{1,n_1}$ $B^1$ $A_{n_1+1,n_2}\dots B^r$ $A_{n_r+1,m}$ $B^{r+1}$  is a non-repetitive sequence.
\end{Corollary}

The following is also
a useful observation about the effect of uniquely coloured vertices: %is the following:

\begin{Lemma} \label{unique}
Let $v$ be a vertex that  obtains a unique colour in a vertex colouring $\varphi$ of a graph $G$.
Then the colouring $\varphi$ is non-repetitive
if and only if
all the components of the graph $G\setminus\{v\}$ are coloured (not facial) non-repetitively.
\end{Lemma}

\begin{proof}
If the colouring of $G$ is non-repetitive, then every path of $G$ is non-repetitive. No more paths are obtained by removing $v$ from $G$ and hence, every path in every component of  $G\setminus\{v\}$  is non-repetitive. \\
On the other hand, assume all the components of the graph $G\setminus\{v\}$ are coloured non-repetitively by $\varphi$. Let $P$ be a path in $G$. If $P$ does not contain the vertex $v$, than $P$ is a path in some component of $G\setminus\{v\}$ and therefore, it is coloured non-repetitively. If $P$ contains a vertex $v$ which is uniquely coloured on $P$ and there is some repetition on $P$, then $v$ is not in the repetitive part of $P$, and the repetition occurs on a subpath of $P$ which belongs to some component of $G\setminus\{v\}$ - a contradiction.
\end{proof}

\textbf{Families with constant upper bounds}\\

By means of a sophisticated application of a special version of the Local Lemma, it was shown in \cite{GPZ10} that $\pi_{\rm ch}(P_n)\leq 4$ for every path $P_n$. It was a significant improvement of the previously known bounds, see \cite{GPZ10} for details.
(Whether the optimal bound in this list version of the original Thue problem is $3$ or $4$ is still an intriguing open problem).
The later proof of the same theorem from \cite{GKM10} was the first (simple) application of the entropy compression method to nonrepetitive sequences.
As a corollary of this result, by (\ref{3piEq}), we immediately obtain the following.

\begin{Corollary}[\cite{GPZ10}] \label{GPZ}
For a path $P_n$ on $n$ vertices, $\pi_{fl} (P_n)\leq 4$.
\end{Corollary}

This result implies the next observation:

\begin{Theorem} \label{cycle}
For a cycle $C_n$ on $n$ vertices, $\pi_{fl}(C_n)\leq 5$.
\end{Theorem}

\begin{proof}
Let $C_n$ be a cycle on $n$ vertices, where each vertex is endowed with a list of at least $5$ admissible colours.
Choose one vertex of $C_n$ and colour it with an arbitrary colour from its list. By removing this colour from the lists of all other vertices (which form a path $P_{n-1}$ on $n-1$ vertices) we make this colour unique in the colouring of $C_n$. By Lemma~\ref{unique} and Corollary~\ref{GPZ}, $C_n$ can be coloured non-repetitively using colours from the assigned lists.
\end{proof}

Very similar arguments give us results for wheels, stars and subdivisions of stars. \\
A {\em star }($S_n$) on $n+1$ vertices is the graph that arises from connecting one vertex ({\em central vertex}) to $n$ vertices of an independent set. A {\em wheel graph} ($W_n$) (or simply wheel) on $n+1$ vertices is the graph formed by connecting a single vertex ({\em central vertex}) to $n$ vertices of a cycle.

\begin{Theorem}
For a star $S_n$ on $n+1$ vertices, $\pi_{fl}(S_n)=2$. \\
For $S$ being a subdivision of a star $S_n$ on $n+1$ vertices, $\pi_{fl}(S)\leq 5$. \\
For a wheel $W_n$ on $n+1$ vertices, $\pi_{fl}(W_n)\leq 6$.
\end{Theorem}

\begin{proof}
Let $G$ belong to one of the families of graphs mentioned in the theorem. Let each vertex of $G$ be endowed with a list of at least $k$ admissible colours ($k=2$ for $G$ being a star $S_n$, $k=5$ for $G$ being a subdivision $S$ of a star $S_n$ and $k=6$ for $G$ being a wheel $W_n$).
Let $v$ be the central vertex of $G$. Colour it with some colour from its list. Remove this colour from the lists of all other vertices of $G$.
According to Lemma \ref{unique} it is sufficient to show that every component of $G':=G\setminus\{v\}$ can be coloured non-repetitively using the remaining colours in the lists.\\
If $G$ is a star then all components of $G'$ are single vertices, so one remaining  colour in each list is enough. If $G$ is a subdivision of a star, the components of $G'$ are paths and 4 remaining colours are enough, and if $G$ is a wheel, then $G'$ is a cycle and 5 remaining colours in each list are sufficient.
\end{proof}

\begin{Theorem}
Let $W$ be a subdivision of a wheel $W_n$. Then $\pi_{fl}(W)\leq 8$.
\end{Theorem}

\begin{proof}
Let $W$ be a subdivision of a wheel $W_n$ and let each vertex of $W$ be endowed with a list of at least $8$ colours.
Let $v$ be the central vertex of  $W_n$ ($W$ respectively). This vertex is uniquely determined if $n\ge 4$. If $n=3$ choose an arbitrary one. Let $v_1=v_{n+1}$, $v_2,\dots,$ $v_n=v_0$ be the remaining vertices of $W_n$ in clockwise order in which they occur on the cycle $W_n-v$, and let $P_1=P_{n+1}$, $P_2,\dots$, $P_n=P_0$ be paths of $W$ starting in $v$ and ending in $v_i$, $i=1,2,\dots,n$, each of them containing no further vertices of $W_n$. For $i=1,2,\dots,n$ denote by $u_i$ the vertex adjacent to $v_i$ on the path $P_i$ different from $v$ (if such a vertex exists). \\
Colour $v$ with some colour from its list. Remove this colour from the lists of all other vertices different from $v_1$, $v_2,\dots,$ $v_n$.
Consider the vertices $u_i$, $i=1,2,\dots,n$, in a clockwise order (where some vertices $u_i$ may be missing if $P_i$ consists only of 2 vertices). Colour these vertices using colours from their lists in such a way that these vertices in the cyclic order around $v$ are coloured non-repetitively. This is possible because each vertex $u_i$ still has a list of at least 7 admissible colours. For $i=1,2,\dots, n$ remove the color of vertex $u_i$ from  the list of every non-coloured vertex on the paths $P_{i-1}$, $P_i$, $P_{i+1}$, on the path between $v_{i-1}$ and $v_i$  containing no further  vertex of $W_n$ and on the path between $v_{i}$ and $v_{i+1}$  containing no further vertex of $W_n$.
From the list of each $v_i$, $i=1,2,\dots,n$, adjacent to $v$ in $W$ remove the colour of the vertex $v$. \\
This way at most $3$ colours are removed from the list of each vertex on the outer cycle of $W$. Hence by Theorem~\ref{cycle} it can be coloured non-repetitively using only the remaining colours from the lists. At most $4$ colours are removed from the list of each vertex on the path $P_i\setminus\{v,v_i,u_i\}$, $i=1,2,\dots,n$. Hence by Corollary~\ref{GPZ} each such a path can be coloured non-repetitively using the remaining colours from the lists. By Corollary~\ref{2} of Lemma~\ref{1} no facial path in $W$ is repetitive.
\end{proof}

\begin{Theorem}\label{product}
Let $G$ be a Cartesian product of two path graphs, $G = P_n\square P_m$, where $n\geq 3$, $m\geq 3$. Then $\pi_{fl} (G)\leq 5$.
\end{Theorem}

\begin{proof}
Let $G$ be a graph described above and let each vertex of $G$ be assigned a list of at least $5$ colours.
Let $v$ be a vertex on the outer cycle of $G$ of degree $3$ adjacent to a vertex $w$ of degree $2$ and let $u$ be a vertex of degree $4$ adjacent to $v$. Let $w$ be adjacent to $z$, $z\neq v$. Colour $v$ with some colour from its list and remove this colour from the lists of all other vertices on the outer cycle $C$ of $G$, so that each vertex on the outer cycle is assigned a list of at least $4$ colours. Colour the path $P=C\setminus\{v\}$ non-repetitively using colours of the remaining lists. Consider the path $Q$ on all the remaining uncoloured vertices of $G$ ending in vertex $u$ and colour these vertices in the order as they appear on $Q$ so that $u$ is coloured as the last one. Let $x\in Q$ be a uncoloured vertex different from $u$. Vertex $x$ (of degree $4$) is adjacent with at most $3$ coloured vertices.
Remove the colours of these vertices from the list of $x$. If less then $3$ colours have been removed and there are two equally coloured neighbors of $x$ on the same facial path, remove from the list of $x$ also the colour of the vertex not adjacent with $x$ but lying also on this facial path (when this vertex is already coloured). Observe that in each case at most $3$ new colours are removed from the list of $x$. Colour $x$ with an arbitrary colour remaining in its list. The sequences of colours on the $4$-gonal faces around $x$ are of the form $ABCD$ or $ABAC$ hence no repetition is obtained.
Now colour the vertex $u$ using similar rules as for colouring the vertex $x$. All the 4 neighbors of $u$ are already coloured but at least two of them by different colours ($v$ and $z$). Hence no more then $4$ colours have to be removed from the list of $u$ and colouring it with an arbitrary remaining colour from its list ensures that the obtained colouring will be facially non-repetitive.
\end{proof}

Using similar arguments as in the proof of Theorem~\ref{product} (with a small variation that those vertices that are not lying at the outer cycle are being colored in arbitrary order - the existence of the path $Q$ is not needed) one can prove that for $G$ being a planar graph which is a large enough subgraph of the regular hexagonal, triangular respectively quadrangular tiling of the plane in the sense that the outer cycle does not have a chord, the facial Thue choice number is not greater then $8$, $7$ respectively $6$. \\
The situation becomes a little bit more complicated in the case of those graphs, where the outer cycle of a graph has a chord.

A {\em ladder graph} $L_n$ is a planar undirected graph with $2n$ vertices and $n+2(n-1)$ edges.
It can be obtained as the Cartesian product of two path graphs, one of which has only one edge: $L_n = P_n\square P_2$.

\begin{figure}[h!]
\begin{center}
    \includegraphics[width=5.7cm]{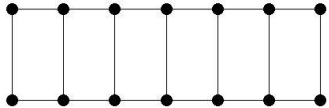}
\caption{Ladder graph}
\end{center}
\end{figure}

\begin{Theorem}\label{ladder}
Let $L_n$ be a ladder graph. Then $\pi_{fl}(L_n)\leq 8$.
\end{Theorem}

\begin{proof}
Let $L_n$ be a ladder graph where each of the vertices is endowed with a list of  at least $8$ admissible colours. Let $u_1$, $v_1$, $u_n$ and $v_n$ be the vertices of degree $2$ in $L_n$ and $v_1$, $v_2,\dots,$ $v_n$, $u_n$, $u_{n-1},\dots,$ $u_1$ the sequence of vertices of $L_n$ in the order in which they appear on its outer cycle. Colour $u_n$ and $v_1$ with two different colours from their lists. Remove these colours from the lists of all other vertices of $L_n$. Colour the path on vertices $v_2$, $v_3,\dots,$ $v_n$ non-repetitively. For $i=1,2,\dots,n-1$ remove the colours of $v_{i}$ and $v_{i+1}$ from the list of colours assigned to $u_i$. Colour the path on vertices $u_1$, $u_2,\dots,$ $u_{n-1}$ non-repetitively with the remaining (at least $4$) colours in each of the lists. By Lemma~\ref{1} and Theorem~\ref{cycle} the outer cycle is coloured non-repetitively. As each sequence of colours on the $4$-gonal faces is of the form $ABCD$ or $ABAC$ they are also coloured non-repetitively.
\end{proof}

By adding two or more edges to a ladder graph $L_n$ some other families of graphs with constant upper bound for the facial Thue choice number can be obtained.

\begin{figure}[h!]
\begin{center}
   \includegraphics[width=5.7cm]{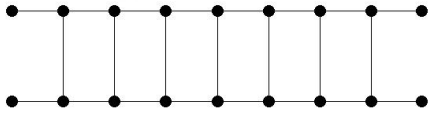}
 \caption{Ladder graph with 4 added edges}
\end{center}
\end{figure}

\begin{Theorem}
Let $L$ be a graph obtained from the ladder graph $L_{n}$ by adding $1$ pending edge to each of the $4$ vertices of degree $2$ in $L_{n}$. Then $\pi_{fl}(L)\leq 6$.
\end{Theorem}

\begin{proof}
Let $L$ be a graph as described above. Let $u_1$, $v_1$, $u_n$ and $v_n$ be the vertices of degree $2$ in $L_n$ and $v_1$, $v_2,\dots,$ $v_n$, $u_n$, $u_{n-1},\dots,$ $u_1$ - the sequence of vertices of $L_n$ in the order in which they appear on its outer cycle. Let the $4$ pending edges added to $L_n$ in order to construct $L$ be $u_0 u_1$, $v_0 v_1$, $u_n u_{n+1}$ and $v_n v_{n+1}$.
Colour the path on vertices $v_0$, $v_1,\dots,$ $v_{n+1}$ non-repetitively. For $i=0,1,\dots,n$ from the list of colours assigned to $u_i$ remove the colour of the vertex $v_{i}$ and $v_{i+1}$. Colour the path on vertices $u_0$, $u_1,\dots,$ $u_{n+1}$ non-repetitively with the remaining (at least $4$) colours in each of the lists.
It is easy to see that obtained colouring is facial non-repetitive.
\end{proof}

A {\em circular ladder graph} $D_n$ ({\em prism graph}) is a planar graph consisting of $2n$ vertices and $3n$ edges and it can be obtained from ladder graph $L_n$ by joining the four vertices of degree $2$ by two edges (as in the figure below). \\

\begin{figure}[h!]
\begin{center}
     \includegraphics[width=4.7cm]{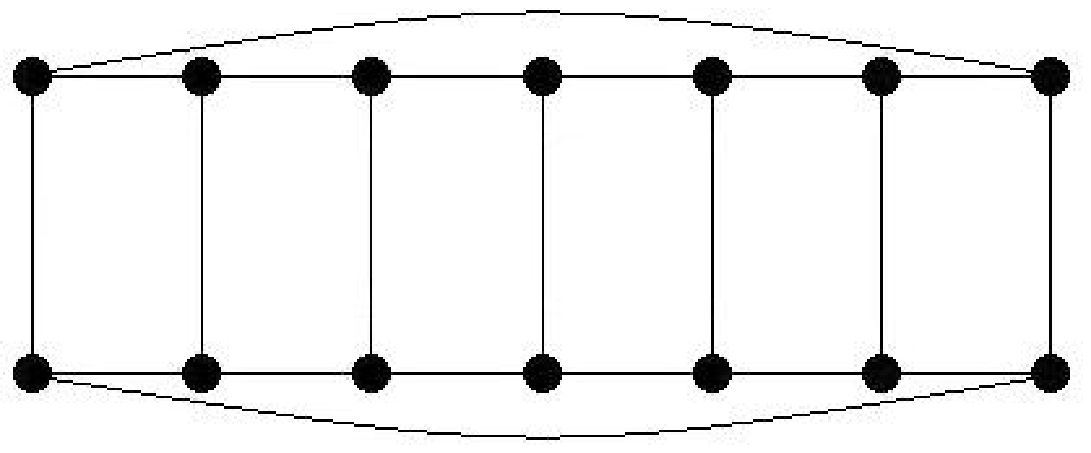}
    \caption{Circular ladder graph}
\end{center}
\end{figure}

\begin{Theorem}
For a prism $D_n$ on $2n$ vertices, $\pi_{fl} (D_n)\leq 8$.
\end{Theorem}

\begin{proof}
We use the same colouring as in the previous proof. There are two $n$-gonal faces of $D_n$ - each consisting of a uniquely coloured vertex and a nonrepetitive path. Hence no one of these contains a repetitive path. Each of the 4-gonal faces has the same property as in the theorem above.
\end{proof}

An {\em apic graph} $A\sp{n}$ is a graph on $2n$ vertices, $n>2$, consisting of $2(n-1)$ $3$-gonal faces and one $2n$-gonal face that can be obtained from the ladder graph $L_n$ by adding one edge to each quadrangle so that no vertex in the resulting graph $A\sp{n}$ is of degree $>4$.

\begin{figure}[h!]
\begin{center}
      \includegraphics[width=5.7cm]{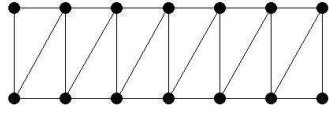}
    \caption{Apic graph}
\end{center}
\end{figure}

\begin{Theorem}
Let $A\sp{n}$ be an apic graph. Then $\pi_{fl} (A^n)\leq 8$.
\end{Theorem}

\begin{proof}
Let $A\sp{n}$ be a graph as described above and let lists of colours of lengths at least $8$ be assigned to all its vertices.
Let $v_1$ and $u_n$ be the vertices of degree $2$ in $A^n$ and $v_1$, $v_2,\dots,$ $v_n$, $u_n$, $u_{n-1},\dots,$ $u_1$ the sequence of vertices of $A^n$ in the order in which they appear on its outer cycle. Colour $u_n$ and $v_1$ with two different colours from their lists. Remove these colours from the lists of all other vertices of $A^n$. Colour the path on vertices $v_2$, $v_3,\dots,$ $v_n$ non-repetitively. For $i=2,3,\dots,n-1$, from the list of colours assigned to $u_i$ remove the colour of the vertex $v_{i}$ and $v_{i-1}$.
Colour the path on vertices $u_1$, $u_2,\dots,$ $u_{n-1}$ non-repetitively with the remaining (at least $4$) colours in each of the lists. By Lemma~\ref{1} and Theorem~\ref{cycle} the outer cycle is coloured non-repetitively. As $3$-gonal faces are also coloured with $3$ different colours, the obtained colouring is facial  non-repetitive.
\end{proof}

An {\em $n$-sided antiprism graph} $A_n$ on $2n$ vertices and $4n$ edges is a graph of polyhedron composed of two parallel copies of some particular $n$-sided polygon, connected by an alternating band of triangles. For $n>4$ it can be also obtained from circular ladder graph $D_n$ by adding one edge to each quadrangle so that no vertex in the resulting graph $A_{n}$ is of degree $>4$.

\begin{figure}[h!]
\begin{center}
    \includegraphics[width=4.7cm]{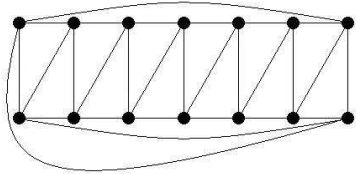}
    \caption{Graph of $n$-sided antiprism}
\end{center}
\end{figure}

\begin{Theorem}
For an antiprism $A_n$ on $2n$ vertices, $\pi_{fl} (A_n)\leq 7$.
\end{Theorem}

\begin{proof}
Let $A_n$ be an antiprism graph where to each of the vertices is assigned a list of colours of length at least $7$.
Let $v_1$, $v_2,\dots,$ $v_n$ be the vertices of the outer cycle of $A_n$ and $u_1$, $u_2,\dots,$ $u_n$ be the vertices of inner cycle of $A_n$, such that $u_1$ is adjacent to $v_1$ and $v_n$ and $u_i$ is adjacent to $v_i$ and $v_{i-1}$ for $i=2,3,\dots,n$. Colour the cycle $v_1,...,v_n$ non-repetitively using colours from the corresponding lists. For every vertex $u_i$ remove the colours of the two adjacent $v$-vertices from its list and colour the cycle $u_1,...,u_n$ non-repetitively using colours from the remaining lists (of length at least 5). In this way, the two $n$-gonal faces are coloured non-repetitively and all triangle faces are rainbow.
\end{proof}

\section{Concluding remarks}

Note that the bound in Theorem \ref{mainresult} is `somewhere in between' the bounds from Theorems~\ref{DujmobicetalTh} and~\ref{BaratCzapTh}, cf. (\ref{3piEq}), and is quite away from the constant upper bound for the non-list version of the same problem from Theorem~\ref{BaratCzapTh}.
We believe however that the following statement is true and pose it as a conjecture (verified already for several classes of graphs in Section~\ref{families_section} above).

\begin{Conjecture}
There exists a constant $C$ such that $\pi_{fl}(G)\le C$ for all plane graphs $G$.
\end{Conjecture}

We recall once more that by the result from \cite{MendezZhu}, a similar conjecture but for $\pi_{\rm ch}$ is far from being true,
as it would be false already for trees. On the other hand, if we consider edge colourings instead of the vertex ones,
then the parameters corresponding to $\pi_{f}(G)$ and $\pi_{fl}(G)$ are denoted by $\pi'_{f}(G)$ and $\pi'_{fl}(G)$
and called the \emph{facial Thue chromatic index} and
the {\em facial Thue choice index}, respectively, see \cite{HJSS09} and \cite{JSES11} for details.
It is known that $\pi'_{f}(G)\leq 8$ (\cite{HJSS09}) and $\pi'_{fl}(G)\leq 12$ (\cite{JP12}), and it is believed that the both parameters are the same for every plane graph $G$, see \cite{JSES11}. However, quite frequently the edge colouring and vertex colouring parameters `tend to behave differently' in list settings.

\end{document}